\documentclass[12pt]{amsart}
\usepackage{amsmath}
\usepackage{amssymb}
\usepackage{amsfonts}
\usepackage{textcomp}
\usepackage{amsthm}
\usepackage{mathrsfs}
\usepackage[latin1]{inputenc}
\usepackage[all]{xy}
\usepackage{hyperref}
\usepackage{url}
\usepackage{graphicx}
\usepackage{tikz-cd}
\usepackage[paperheight=11in,paperwidth=8.5in,left=1in,top=1.25in,right=1in,bottom=1.25in,]{geometry}
\usepackage{enumitem}

\newcommand{\HH}{\mathbb{H}}
\newcommand{\RR}{\mathbb{R}}

\newcommand{\I}{\rm{I}}
\newcommand{\II}{{\rm I}\!{\rm I}}
\newcommand{\VR}{{\rm V}_{\rm R}}

\newcommand{\Eps}{{\rm Eps}}

\allowdisplaybreaks

\newtheorem{theorem}{Theorem}[section]
\newtheorem{maintheorem}{Theorem}[section]

\newtheorem{defi}{Definition}[section]

\newtheorem{remark}{Remark}[section]

\title{W-volume for planar domains with circular boundary}
\author{Jeffrey Brock and Franco Vargas Pallete}
\thanks{Work of the first author was supported by NSF grant DMS-1849892. The second author was partially supported by NSF grant DMS-2001997}

\address{Department of Mathematics  \\
 Yale University \\
New Haven, CT 06511\\
U.S.A.}
\address{Department of Mathematics  \\
 Yale University \\
New Haven, CT 06511\\
U.S.A.}
\begin{document}

\maketitle

\begin{abstract}
    We extend the notion of Epstein maps to conformal metrics on submanifolds of the unit sphere $\mathbb{S}^n=\partial_\infty\mathbb{H}^{n+1}$. Using this construction for curves in $\mathbb{S}^2$, we define the $W$-volume for conformal metrics on domains in $\overline{\mathbb{C}}=\mathbb{S}^2$ with round circles as boundaries. We show that the $W$-volume is a realization in $\mathbb{H}^3$ of the determinant of the Laplacian. We use this and work of Osgood, Phillips and Sarnak to show that a classical Schottky uniformization of a genus $g$ Riemann surface has renormalized volume bounded by $(6g-8)\pi$, and by $-2\pi$ under further assumptions. This gives a partial answer to a question of Maldacena. We also then provide a $\mathbb{H}^3$ realization of the Loewner energy of a $C^{2,\alpha}$ Jordan curve.
\end{abstract}

\section{Introduction}

Given a conformally compact hyperbolic $3$-manifold $(M,h)$ and a metric $g$ on its conformal boundary, we can define the functional $W$-volume as follows (\cite{Graham}, \cite{KrasnovSchlenker08}). Associated to $g$ there exists a equidistant foliation $\lbrace N_r\rbrace_{r\gg0}$ in $M$ expanding towards the conformal boundary $\partial M$ so that $\lim_{r\rightarrow\infty} 4e^{-2r}h|_{N_r}=g$. The $W$-volume of $(M,h,g)$ can be defined as

\begin{equation}\label{eq:OGWvol}
    W(M,h,g) := vol(int(N_r)) - \frac12 \int_{N_r}Hda + r\pi\chi(\partial M),
\end{equation}
where $int(N_r)$ is the region bounded by $N_r$, $H$ is the mean curvature of $N_r$ with respect to the normal vector pointing towards $int(N_r)$, $da$ is the area element and $\chi(.)$ denotes Euler characteristic. It is part of the definition of $W$-volume to verify that the right-hand side of (\ref{eq:OGWvol}) does not depend on $r$. We will simplify notation by writing $W(g)$, as $(M,h)$ will be understood and for most it unchanged.

It this setup, it has been noted (see for instance \cite[Proposition 3.11]{GuillarmouMoroianuSchlenker}) that $W$-volume and the well-studied functional $\log\det'\Delta$ (the logarithm of the determinant of the Laplace-Beltrami operator, see \cite{OsgoodPhillipsSarnak}) agree on conformal metrics by an affine relation, namely:

\begin{equation}\label{eq:affinerelation}
W(e^{2\varphi}g) - W(g) = 3\pi (\log{\rm det}'\Delta(e^{2\varphi}g) - \log{\rm det}'\Delta(g)),
\end{equation}
for any conformal metric $g$ and any smooth function $\varphi:\partial M \rightarrow \mathbb{R}$.

The notion of $W$-volume can be used to define {\em renormalized volume} (denoted by $\VR$), as $\VR$ can be obtained by evaluating $W$ at the (unique) hyperbolic metric in the conformal boundary $\partial M$.  The Renormalized Volume is a geometric quantity motivated by the AdS/CFT correspondence and the calculation of gravitational action, as described by Witten and Graham (\cite{Witten98}, \cite{Graham}), which has seen increased interest through a range of applications to the study of 3-dimensional hyperbolic manfiolds and their deformation spaces.

In this article we extend the definition of $W$-volume for regions with round boundaries in $\mathbb{S}^2$, so that the relation (\ref{eq:affinerelation}) remains true . Namely, given $U\subseteq \overline{\mathbb{C}}$ with $\partial U$ a union of finitely many round circles and $e^{2\varphi}|dz|^2$ conformal metric in $U$, we define a region $B(\varphi)\subset \mathbb{H}^3$ with piecewise smooth boundary $\partial B(\varphi)$ and a distinguished subset $C(\varphi)\subset \partial B(\varphi)$, so that

\begin{maintheorem}\label{thm:main}
Under the previous conditions, if we define
\[
    W(e^{2\varphi}|dz|^2) := vol(B(\varphi)) - \frac12 \int_{\partial B(\varphi)} Hda - \frac34 Area(C(\varphi)),
\]
then $W(.)$ satisfies the affine relation (\ref{eq:affinerelation}) with $\log\det\Delta$.
\end{maintheorem}

It should be noted that if $\omega\in {\rm PSL}(2,\mathbb{C})$, then the $B$, $C$ regions for $V=\omega^{-1}U$, and $e^{2\psi}|dz|^2 = \omega^*(e^{2\varphi}|dz|^2)$ will satisfy $$\omega(B(V,\psi)) = B(U\varphi), \ \ \ \text{and} \ \ \ \omega(C(V,\psi)) = C(U,\varphi).$$ In particular we will have $W(U, e^{2\varphi}|dz|^2) = W(V, e^{2\psi}|dz|^2)$. In short, this construction is ${\rm PSL}(2,\mathbb{C})$ invariant.

Planar domains with round circles as boundary are an interesting class to study and enjoy good analytical properties while considered as the boundary of regions in $\mathbb{H}^3$, as done for instance in \cite{PhillipsSarnak85} for the study of the Laplacian and limit sets of Kleinian groups.  

While realizing $\log\det\Delta$ in $\mathbb{H}^3$ in this manner is of independent interest, we provide two applications of this realization.

Maldacena \cite{Maldacena} asked, for a given Riemann surface $\Sigma$, how to minimize $\VR$ on hyperbolic fillings of $\Sigma\cup \Sigma^*$ (i.e. hyperbolic manifolds $M$ so that the boundary at infinite $\partial_\infty M$ coincides with $\Sigma$, $\Sigma^*$ after forgetting markings), where $\Sigma^*$ is the Riemann surface obtained by reversing the orientation of $\Sigma$. Since we can fill in $\Sigma\cup\Sigma^*$ by either taking $M=\Sigma\times\mathbb{R}$ or by taking $M$ to be the union of two disjoint handlebodies, then we are interested in finding which of the following quantities is the smallest:

\begin{enumerate}[label=(\alph*)]
    \item\label{case:QF} $\inf_{M\in QF(\Sigma)} \VR(M)$, where $QF(\Sigma)$ is the set of quasi-Fuchsian manifolds $M=\Sigma\times\mathbb{R}$ whose boundary at infinity is conformal to $\Sigma\cup\Sigma^*$.
    \item\label{case:Schottky} $\inf_{M\in S(\Sigma)} 2\VR(M)$, where $S(\Sigma)$ is the set of Schottky manifold $M$ (namely, a handlebody) with boundary at infinity is conformal to $\Sigma$.
\end{enumerate}

As it is known that in case \ref{case:QF} the smallest volume is $0$ (\cite{BridgemanBrockBromberg}, \cite{FVP17}, \cite{BridgemanBrombergVP}), then Maldacena's question reduces to prove whether we can produce a Schottky manifolds with prescribed conformal boundary and negative renormalized volume. Previous work on this question has been done by the second author in the presence of short geodesics (see \cite{FVP19}). Using a (still open) conjecture of Bers that states that any Riemann surface has a classical Schottky uniformization (one where there exists a fundamental domain at $\overline{\mathbb{C}}$ with round circles as boundary) we prove that an upper bound for case \ref{case:Schottky}, and that upper bound is negative under further assumptions. Namely, we prove the following.

\begin{maintheorem}\label{thm:main1}
Let $M$ be the hyperbolic handlebody obtained by the classical Schottky uniformization of a surface $\Sigma$ of genus $g$. Then $\VR(M) < (6g-8)\pi$. Moreover, if the distance between geodesic planes on the boundary of a fundamental domain for $M$ is at least $4$, then $\VR(M)<-2\pi$.
\end{maintheorem}

The second application (observed to us by Yilin Wang) is that we can use our definition of $W$-volume to give a realization of the Loewner energy in $\mathbb{H}^3$. Loewner energy is a well-studied potential for Weil--Petersson Jordan curves and it is a potential of the Weil-Petersson metric, namely that its Hessian at the round circle equals a multiple of the Weil--Petersson metric. Our result allows us to present a new formulation for this potential if the Jordan curve is sufficiently smooth.

\begin{maintheorem} \label{thm:main2}
Let $\gamma$ be a $C^{2,\alpha}$ Jordan curve, and let $g_{1,2}$ be conformal metrics in $\mathbb{D}, \mathbb{D}^*$ (respectively) given as the pull-back of the round metric of $\mathbb{S}^2$ by uniformizing the complement of $\gamma$. Then the Loewner energy $I^L(\gamma)$ can be computed as

\begin{equation}
I^L(\gamma) = -\frac{4}{\pi} (W(g_1)+W(g_2)).
\end{equation}
In particular, for $C^{2,\alpha}$ Jordan curves one can express a K\"ahler potential for the Weil--Peterson metric in terms of $W$-volume.
\end{maintheorem}

\subsection*{Outline} In Section \ref{sec:Epsteinsurfaces} we discuss our extension of Epstein maps.  In Subsection \ref{subsecGeneralEps} we describe the construction in general dimension, while in Subsection \ref{subsec:2dEps} we carry out explicit computations in $\mathbb{H}^3$. Then in Subsection \ref{subsec:W-vol} we define $W$-volume for conformal metrics in domains with round boundaries and proves Theorem \ref{thm:main}. Section \ref{sec:applications} contains our applications. In Subsection \ref{subsec:negVR} we prove Theorem \ref{thm:main1}, while in Subsection \ref{subsec:Loewner} we prove Theorem \ref{thm:main2}.

\section*{Acknowledgements}
We would like to thank Juan Maldacena for bringing this problem to our attention, as well as for our much appreciated discussions about the topic. We are also thankful to Keaton Quinn, Peter Sarnak and Yilin Wang for helpful discussions and comments, as well as to Ken Bromberg for pointing out a gap on an earlier version of Theorem \ref{thm:main1}.

\section{Epstein maps}\label{sec:Epsteinsurfaces}

\subsection{Epstein maps for general submanifolds}\label{subsecGeneralEps}

Let $M^k \subset \mathbb{R}^{n}$ be a k-dimensional submanifold and let $\varphi: M\rightarrow \RR$ be a smooth function. For any $p\in M$ we can consider the horosphere $H_{p, \varphi(p)}$ tangent to $\partial_{\infty}\HH^{n+1} = \mathbb{R}^{n}$ at $p$ and that has signed distance from $(0,\ldots,0,1)$ equal to $\varphi(p)$. Hence $H(p,\varphi)$ is parametrized by 

\[H(p,\varphi) = \big\lbrace (p,e^{-\varphi(p)}) + e^{-\varphi(p)}Y \,\big|\, Y\in \mathbb{S}^n \big\rbrace,
\]
where the point at infinity corresponde to $Y=(0,\ldots,0,-1)$. This horosphere can be found as the solution of points $P\in\mathbb{H}^{n+1}$ so that if $\eta$ is an isometry of $\mathbb{H}^n$ satisfying $\eta(P)=(0,\ldots,0,1)$ then $\eta^*(g_{\mathbb{S}^{n}})(p,0) = e^{\phi(p)}g_{\mathbb{S}^{n}}(p,0)$, where $g_{\mathbb{S}^n}$ denotes the round metric in $\mathbb{S}^n=\RR^{n}\cup\lbrace\infty\rbrace$. Hence we think of $e^\phi$ as a conformal factor with respect to the spherical metric. 

By $\Eps_{\varphi,M}$ we denote the \emph{envelope} of those horospheres, also known as the \emph{Epstein map} associated to $\varphi, M$. More precisely, we will define $\Eps_{\varphi,M}:T^\perp M \rightarrow \HH^{n+1}$ (where $T^\perp$ denotes the normal bundle of $M$) as the smooth map given by the \emph{envelope restrictions}

\begin{enumerate}
    \item\label{item:tangency} $D_{(p,x)}\Eps_{\varphi,M}$ has image contained in $T_{\Eps_{\varphi,M}(p,x)}H_{p,\varphi(p)}$
    \item\label{item:horosphere} $\Eps_{\varphi,M}(T^\perp_pM)$ parametrizes a $(n-k)$-dimensional horosphere in $H_{p,\varphi(p)}$
\end{enumerate}
For the cases where the submanifold $M$ is understood, we will drop it from the notation of $\Eps$. We also define the Gauss map $\widetilde{\Eps}_{\varphi}:T^\perp M \rightarrow T^1\HH^{n+1}$ as the outer-normal to $H_{p,\varphi(p)}$ at $\Eps_{\varphi}(p,x)$. While in general $\Eps_\varphi$ might not be an immersion, we will see that $\widetilde{\Eps}$ is always an embedding.

Following \cite{Epstein84} we need then to solve for a function $Y:T^\perp M \rightarrow \mathbb{S}^n$ so that 
\begin{equation}\label{eq:Epsteinmap}
\Eps(p,x) = (p,e^{-\varphi(p)}) + e^{-\varphi(p)}Y(p,x)
\end{equation}
solves (\ref{item:tangency}) and (\ref{item:horosphere}). As done by Epstein, we will show that $Y$ solves a particular linear system of equations.

Consider $v_1, \ldots, v_k$ an orthonormal basis of $T_pM$ and $e_1, \ldots, e_{n-k}$ an orthonormal basis of $T^\perp_p M$. Then the tangent space of $\Eps$ is generated by

\begin{equation}\label{eq:envelopedirections}
\Eps_{v_i} = (v_i, -\varphi_{v_i}e^{-\varphi(p)}) - \varphi_{v_i}e^{-\varphi(p)}Y + e^{-\varphi(p)} Y_{v_i}, \quad i=1, \ldots, k
\end{equation}

\begin{equation}\label{eq:horospheredirections}
\Eps_{e_l} = e^{-\varphi(p)}Y_{e_l}, \quad l=1, \ldots, n-k
\end{equation}

Similarly, the tangent space to the horosphere $H(p,\alpha)$ is given by the orthogonal complement of $Y\in \mathbb{S}^n \subset \RR^{n+1}$. $\Eps$ is then the solution to the equations defined by making the tangent spaces to coincide. Hence multiplying (\ref{eq:envelopedirections}), (\ref{eq:horospheredirections}) by Y the previous equations and using $Y.Y=1, Y.Y_{v_i}= Y.Y_{e_\ell} = 0,\,i=1,\ldots,k, l=1, \ldots, n-k$ we obtain the linear system

\begin{equation}\label{eq:linearsystem}
\langle (v_i, -\varphi_{v_i}e^{-\varphi(p)}) , Y \rangle = \varphi_{v_i}e^{-\varphi(p)},\quad i=1,\ldots, k
\end{equation}
since the equations from (\ref{eq:horospheredirections}) are always satisfied.

Denote by $Z$ the space of solutions $Y$ of the linear system of equations $\langle (v_i, -\varphi_{v_i}e^{-\varphi(p)}) , Y \rangle =0$ for $i=1,\ldots, k$. Since the entries $v_i$ are orthonormal, we have that $Z$ is $(n-k+1)$-dimensional. Moreover, we have that $(e_1,0),\ldots, (e_{n-k},0), (e^{-\varphi}\nabla_p\varphi,1)$ is a basis of $Z$, where $\nabla_p$ is the gradient of $\phi$ with respect to the metric in $M$ induced by $\RR^n$

One can easily verify, $Y=(0,\ldots,0,-1)$ is a solution to the system \ref{eq:linearsystem}. We can see this geometrically by recalling that the plane $\mathbb{R}^n\times\lbrace0\rbrace = \partial_{\infty}\HH^{n+1}$ is tangent to any horosphere based at $M$. We can then characterize the space of solutions of (\ref{eq:linearsystem}) by the points of norm $1$ in $Z+(0,\ldots,0,-1)$. Geometrically, we are translating the kernel $Z$ so it passes through $(0,\ldots,0,-1)$ and we are considering its intersection with the unit sphere $\mathbb{S}^n$. This intersection (denoted by say $A(p)$) is a round $k$-dimensional sphere in $\mathbb{S}^n$, and we are interested in all the points of $A(p)$ distinct to $(0,\ldots,0,1)$. We wish to parametrize $A(p)\setminus\lbrace p\rbrace$ by $T^\perp_pM$ in such a way that $Y(\alpha(p,x))$ varies continuously on $(p,x)$, to then define $\Eps$ by the formula (\ref{eq:Epsteinmap}). We will use \cite{Epstein84} in the following way to achieve our goal. In \cite{Epstein84}, Epstein formula for $Y$ depends on the 1-jet of a conformal metric, meaning that the formula is found in terms of the values of $\phi$ and $\nabla\phi$ at a point (where $\nabla$ denotes the gradient with respect to the euclidean metric of $\RR^n$), and it is found by solving a linear system of equations similar to \eqref{eq:envelopedirections}. For the upper-half space model this solution is given by 

\begin{equation}\label{eq:EpsteinOGSolution}
    Y = \left(\frac{2e^{-\phi}\nabla\phi}{1+e^{-2\phi}|\nabla \phi|^2}, \frac{1-e^{-2\phi}|\nabla\phi|^2}{1+e^{-2\phi}|\nabla \phi|^2} \right)
\end{equation}

At any point $p\in M$, we know the value of $\phi$ and we know the projection of $\nabla\phi$ to $T_p M$ as $\nabla_p\phi$. Hence we define $\nabla\phi$ to $\mathbb{R}^n$ so that its projection to $T^\perp_p M$ is equal to $x\in T^\perp_pM$. The solution from \cite{Epstein84} belong to $A(p)\setminus\lbrace p\rbrace$, as it is the solution of a system of equations that contains \eqref{eq:envelopedirections}, and it is clear that expand all points in the punctured sphere $A(p)\setminus\lbrace p\rbrace$.

Observe that every vertical plane $T^\perp_pM$ in $\Eps: T^\perp M \rightarrow \HH^{n+1}$ is mapped isometrically into the $k$-dimensional horosphere $A(p)$ tangent at $p$. The $k$-dimensional horosphere $A(p)=\Eps (T^\perp_p M)$ lies inside the horosphere $H(p,\varphi(p))$. Hence our definition of $\Eps$ satisfies (\ref{item:tangency}) and (\ref{item:horosphere}).

As announced, we can also define $\widetilde{\Eps}_{\varphi}:T^\perp M \rightarrow T^1\HH^{n+1}$ as the outer-normal to $H_{p,\varphi(p)}$ at $\Eps_{\varphi}(p,x)$. As the (negative) endpoint of $\widetilde{\Eps}_{\varphi}$ is a left inverse of $\widetilde{\Eps}_{\varphi}$, we have that the map $\widetilde{\Eps}_{\varphi}$ is an embedding.

\subsection{Epstein surfaces for Jordan curves in the euclidean plane}\label{subsec:2dEps}

We will focus our attention in the case $n=2, k=1$. In particular $M$ is topologically a circle or an interval, and it has trivial normal bundle.

Take then a parametrization $\gamma(s)$ of $M$ by arc-length, so that $\gamma'(s)$ and $i\gamma'(s)$ are unit vectors that generate $T_{\gamma(s)}M$ and $T^\perp_{\gamma(s)} M$, respectively. Parametrize as well the function $\varphi$ on the arc-length variable $s$.

In this parametrization, the kernel $Z$ associated to the linear system (\ref{eq:linearsystem}) is generated by $(i\gamma'(s),0)$ and $(\varphi'(s)e^{-\varphi(s)}\gamma'(s),1)$. Hence we are looking for solutions $u,v$ so that $Y= u(i\gamma'(s),0) + v(\varphi'(s)e^{-\varphi(s)}\gamma'(s),1) + (0,0,-1) = (v\varphi'(s)e^{-\varphi(s)}\gamma'(s) + iu\gamma'(s), v-1)$ in the unit sphere. This reduces to

\begin{align*}
    &(v\varphi'e^{-\varphi})^2 + u^2 + (v-1)^2=1\\
    &v\left(v|\varphi'|^2e^{-2\varphi} + v\left(\frac{u}{v}\right)^2 + v-2\right)=0\\
    &v = \frac{2}{1+|\varphi'|^2e^{-2\varphi} + (u/v)^2}
\end{align*}

Defining the variable $t= \left(\frac{u}{v}\right)e^{\varphi}$ we have that

\begin{align*}
    v&= \frac{2}{1+(|\varphi'|^2+t^2)e^{-2\varphi}} = \frac{2e^{2\varphi}}{e^{2\varphi}+|\varphi'|^2+t^2}\\
    u&= \frac{2te^{-\varphi}}{1+(|\varphi'|^2+t^2)e^{-2\varphi}} = \frac{2te^{\varphi}}{e^{2\varphi}+|\varphi'|^2+t^2}
\end{align*}
and $Y=Y(s,t)$ can be parametrize as

\begin{equation}
    Y = \left( \frac{2\varphi'e^{\varphi}}{e^{2\varphi}+|\varphi'|^2+t^2}\gamma'(s)+ \frac{2te^{\varphi}}{e^{2\varphi}+|\varphi'|^2+t^2}i\gamma'(s), \frac{e^{2\varphi}-|\varphi'|^2-t^2}{e^{2\varphi}+|\varphi'|^2+t^2}\right)
\end{equation}

Consequently, the Epstein map is given by

\begin{equation}\label{eq:Epsteintube}
    \begin{split}
    \Eps(s,t) &= (\gamma(s),e^{-\varphi(s)}) + e^{-\varphi(s)}Y(s,t) \\&= (\gamma(s),0) + \left( \gamma'(s)\frac{2\varphi'}{e^{2\varphi}+|\varphi'|^2+t^2}+ i\gamma'(s)\frac{2t}{e^{2\varphi}+|\varphi'|^2+t^2},  \frac{2e^{\varphi}}{e^{2\varphi}+|\varphi'|^2+t^2}\right)
    \end{split}
\end{equation}

We refer to the parametrized surface $\Eps$ (or portions of it) colloquially as \emph{caterpillar regions}, as its decomposition into horocycles along $M$ resembles the circular pattern of a caterpillar.

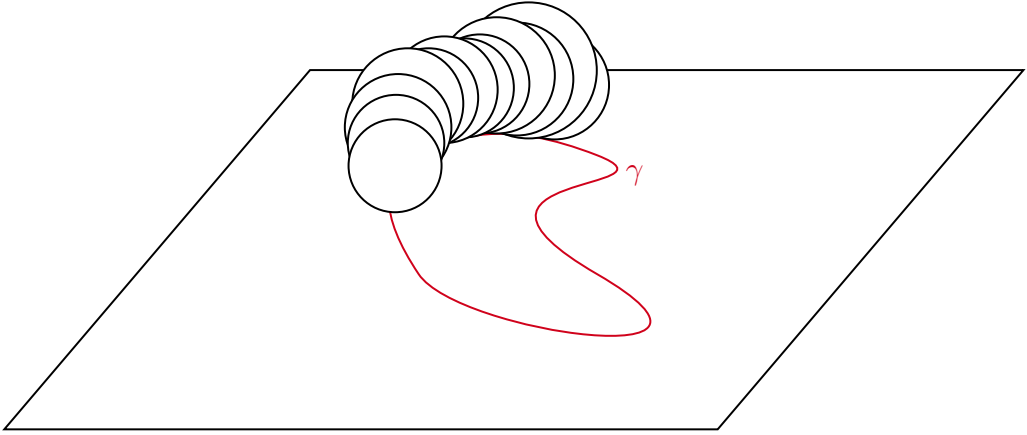
\begin{figure}[hbt!]
    \centering

\tikzset{every picture/.style={line width=0.75pt}} %set default line width to 0.75pt        

\begin{tikzpicture}[x=0.75pt,y=0.75pt,yscale=-1,xscale=1]
%uncomment if require: \path (0,310); %set diagram left start at 0, and has height of 310

%Shape: Polygon Curved [id:ds39414427127729335] 
\draw  [color={rgb, 255:red, 208; green, 2; blue, 27 }  ,draw opacity=1 ] (272,128) .. controls (292,118) and (315,109.67) .. (362,128) .. controls (409,146.33) and (273,136.33) .. (362,188) .. controls (451,239.67) and (292,218) .. (272,188) .. controls (252,158) and (252,138) .. (272,128) -- cycle ;
%Shape: Parallelogram [id:dp1188273877263295] 
\draw   (217.2,85) -- (577,85) -- (422.8,266.33) -- (63,266.33) -- cycle ;
%Shape: Circle [id:dp7351400938274586] 
\draw  [fill={rgb, 255:red, 255; green, 255; blue, 255 }  ,fill opacity=1 ] (313,92.5) .. controls (313,77.31) and (325.31,65) .. (340.5,65) .. controls (355.69,65) and (368,77.31) .. (368,92.5) .. controls (368,107.69) and (355.69,120) .. (340.5,120) .. controls (325.31,120) and (313,107.69) .. (313,92.5) -- cycle ;
%Shape: Circle [id:dp5468795515357792] 
\draw  [fill={rgb, 255:red, 255; green, 255; blue, 255 }  ,fill opacity=1 ] (293.17,85.17) .. controls (293.17,66.2) and (308.54,50.83) .. (327.5,50.83) .. controls (346.46,50.83) and (361.83,66.2) .. (361.83,85.17) .. controls (361.83,104.13) and (346.46,119.5) .. (327.5,119.5) .. controls (308.54,119.5) and (293.17,104.13) .. (293.17,85.17) -- cycle ;
%Shape: Circle [id:dp12379410095042043] 
\draw  [fill={rgb, 255:red, 255; green, 255; blue, 255 }  ,fill opacity=1 ] (293,89.5) .. controls (293,73.76) and (305.76,61) .. (321.5,61) .. controls (337.24,61) and (350,73.76) .. (350,89.5) .. controls (350,105.24) and (337.24,118) .. (321.5,118) .. controls (305.76,118) and (293,105.24) .. (293,89.5) -- cycle ;
%Shape: Circle [id:dp686315326024908] 
\draw  [fill={rgb, 255:red, 255; green, 255; blue, 255 }  ,fill opacity=1 ] (282,87.67) .. controls (282,71.47) and (295.13,58.33) .. (311.33,58.33) .. controls (327.53,58.33) and (340.67,71.47) .. (340.67,87.67) .. controls (340.67,103.87) and (327.53,117) .. (311.33,117) .. controls (295.13,117) and (282,103.87) .. (282,87.67) -- cycle ;
%Shape: Circle [id:dp12401755887072818] 
\draw  [fill={rgb, 255:red, 255; green, 255; blue, 255 }  ,fill opacity=1 ] (277.83,92) .. controls (277.83,78.19) and (289.03,67) .. (302.83,67) .. controls (316.64,67) and (327.83,78.19) .. (327.83,92) .. controls (327.83,105.81) and (316.64,117) .. (302.83,117) .. controls (289.03,117) and (277.83,105.81) .. (277.83,92) -- cycle ;
%Shape: Circle [id:dp27415298956380796] 
\draw  [fill={rgb, 255:red, 255; green, 255; blue, 255 }  ,fill opacity=1 ] (270,94) .. controls (270,80.19) and (281.19,69) .. (295,69) .. controls (308.81,69) and (320,80.19) .. (320,94) .. controls (320,107.81) and (308.81,119) .. (295,119) .. controls (281.19,119) and (270,107.81) .. (270,94) -- cycle ;
%Shape: Circle [id:dp8670052906479482] 
\draw  [fill={rgb, 255:red, 255; green, 255; blue, 255 }  ,fill opacity=1 ] (258,94.92) .. controls (258,80.05) and (270.05,68) .. (284.92,68) .. controls (299.78,68) and (311.83,80.05) .. (311.83,94.92) .. controls (311.83,109.78) and (299.78,121.83) .. (284.92,121.83) .. controls (270.05,121.83) and (258,109.78) .. (258,94.92) -- cycle ;
%Shape: Circle [id:dp7475998941402995] 
\draw  [fill={rgb, 255:red, 255; green, 255; blue, 255 }  ,fill opacity=1 ] (252,99) .. controls (252,85.19) and (263.19,74) .. (277,74) .. controls (290.81,74) and (302,85.19) .. (302,99) .. controls (302,112.81) and (290.81,124) .. (277,124) .. controls (263.19,124) and (252,112.81) .. (252,99) -- cycle ;
%Shape: Circle [id:dp18341441716301454] 
\draw  [fill={rgb, 255:red, 255; green, 255; blue, 255 }  ,fill opacity=1 ] (238.5,102) .. controls (238.5,86.54) and (251.04,74) .. (266.5,74) .. controls (281.96,74) and (294.5,86.54) .. (294.5,102) .. controls (294.5,117.46) and (281.96,130) .. (266.5,130) .. controls (251.04,130) and (238.5,117.46) .. (238.5,102) -- cycle ;
%Shape: Circle [id:dp6941727549484338] 
\draw  [fill={rgb, 255:red, 255; green, 255; blue, 255 }  ,fill opacity=1 ] (234.75,113.88) .. controls (234.75,99.03) and (246.78,87) .. (261.63,87) .. controls (276.47,87) and (288.5,99.03) .. (288.5,113.88) .. controls (288.5,128.72) and (276.47,140.75) .. (261.63,140.75) .. controls (246.78,140.75) and (234.75,128.72) .. (234.75,113.88) -- cycle ;

%Shape: Circle [id:dp06413217971101037] 
\draw  [fill={rgb, 255:red, 255; green, 255; blue, 255 }  ,fill opacity=1 ] (236.25,121.83) .. controls (236.25,108.39) and (247.14,97.5) .. (260.58,97.5) .. controls (274.02,97.5) and (284.92,108.39) .. (284.92,121.83) .. controls (284.92,135.27) and (274.02,146.17) .. (260.58,146.17) .. controls (247.14,146.17) and (236.25,135.27) .. (236.25,121.83) -- cycle ;
%Shape: Circle [id:dp14274699355972453] 
\draw  [fill={rgb, 255:red, 255; green, 255; blue, 255 }  ,fill opacity=1 ] (236.63,133.29) .. controls (236.63,120.34) and (247.13,109.83) .. (260.08,109.83) .. controls (273.04,109.83) and (283.54,120.34) .. (283.54,133.29) .. controls (283.54,146.25) and (273.04,156.75) .. (260.08,156.75) .. controls (247.13,156.75) and (236.63,146.25) .. (236.63,133.29) -- cycle ;

% Text Node
\draw (375,131.4) node [anchor=north west][inner sep=0.75pt]  [color={rgb, 255:red, 208; green, 2; blue, 27 }  ,opacity=1 ]  {$\gamma $};

\end{tikzpicture}

    \caption{Representation of finitely many segments of a caterpillar region over the curve $\gamma$}
    \label{fig:caterpillar}
\end{figure}\noindent

In order to find the normal vector $\widetilde{\Eps}=N$ to the horosphere, we can observe that $N$ is parallel to the vector $Y$. Taking then $N$ so that it is a unit vector in $\mathbb{H}^3$ we obtain

\begin{equation}\label{eq:normalvector}
N(s,t) = \left( \gamma'(s)\frac{4\varphi'e^{\varphi}}{(e^{2\varphi}+|\varphi'|^2+t^2)^2}+ i\gamma'(s)\frac{4te^{\varphi}}{(e^{2\varphi}+|\varphi'|^2+t^2)^2},  \frac{4e^{2\varphi}}{(e^{2\varphi}+|\varphi'|^2+t^2)^2}\right)
\end{equation}

In order to calculate the first and second fundamental forms associated to $\Eps$, we consider the coordinate tangent vector fields $S= \partial_s\Eps,\, T=\partial_t\Eps$. These vector fields are given by the explicit yet laborious formulas

\begin{equation}\label{eq:Tvector}
\begin{split}
    T &= \bigg( \gamma'(s)\frac{-4t\varphi'}{(e^{2\varphi}+|\varphi'|^2+t^2)^2}+ i\gamma'(s)\left(\frac{2}{e^{2\varphi}+|\varphi'|^2+t^2} + \frac{-4t^2}{(e^{2\varphi}+|\varphi'|^2+t^2)^2}\right),\\&  \frac{-4te^{\varphi}}{e^{2\varphi}+|\varphi'|^2+t^2}\bigg)
    \\& = \bigg(\gamma'(s)\frac{-4t\varphi'}{(e^{2\varphi}\gamma'+|\varphi'|^2+t^2)^2} +  i\gamma'(s)\frac{2(e^{2\varphi} + |\varphi'|^2-t^2)}{(e^{2\varphi}+|\varphi'|^2+t^2)^2}, \frac{-4te^{\varphi}}{(e^{2\varphi}+|\varphi'|^2+t^2)^2}\bigg)
\end{split}
\end{equation}

Since $\gamma(s)$ is parametrized by arc-length, we have that $\gamma''(s)=k(s)i\gamma'(s)$, where $k(s)$ is the curvature at $\gamma(s)$.

\begin{equation}\label{eq:Svector}
\begin{split}
    S &= (\gamma'(s),0) + \bigg( \gamma'(s) \frac{2\varphi''(e^{2\varphi}+|\varphi'|^2+t^2) - 2\varphi'(2\varphi'e^{2\varphi} + 2\varphi'\varphi'')}{(e^{2\varphi}+|\varphi'|^2+t^2)^2} + ik\gamma'(s) \frac{2\varphi'}{e^{2\varphi}+|\varphi'|^2+t^2}
    \\& + i\gamma'(s) \frac{-2t(2\varphi'e^{2\varphi}+2\varphi'\varphi'')}{(e^{2\varphi}+|\varphi'|^2+t^2)^2} -k\gamma'(s)\frac{2t}{e^{2\varphi}+|\varphi'|^2+t^2},
    \\& \frac{2\varphi'e^{\varphi}(e^{2\varphi}+|\varphi'|^2+t^2) - 2e^{\varphi}(2\varphi'e^{2\varphi}+2\varphi'\varphi'')}{(e^{2\varphi}+|\varphi'|^2+t^2)^2} \bigg)\\
    &= \bigg(\gamma'(s)\bigg(\frac{2 (e^{2\varphi} + t^2) \varphi'' - 2 |\varphi'|^2 (\varphi'' + 2 e^{2\varphi})}{(e^{2\varphi}+|\varphi'|^2+t^2)^2} - \frac{2t\kappa}{e^{2\varphi} + |\varphi'|^2+t^2} + 1\bigg)
    \\&+ i\gamma'(s)\bigg(\frac{-4t\varphi'(\varphi''+e^{2\varphi})}{(e^{2\varphi}+|\varphi'|^2+t^2)^2}  + \frac{2\varphi'\kappa}{e^{2\varphi} + |\varphi'|^2+t^2} \bigg),
    \\&\frac{-2\varphi'e^{\varphi}(2\varphi''-|\varphi'|^2+e^{2\varphi} -t^2)}{(e^{2\varphi}+|\varphi'|^2+t^2)^2}\bigg)
\end{split}
\end{equation}

Similarly, we differentiate the normal vector from (\ref{eq:normalvector}) with respect to $t$ and $s$ to obtain

\begin{equation}\label{eq:partialTN}
    \partial_tN=\bigg(\gamma'(s)\frac{-16t\varphi'e^{2\varphi}}{(e^{2\varphi}+|\varphi'|^2+t^2)^3}
    +i\gamma'(s)\frac{4e^{2\varphi}(e^{2\varphi}+|\varphi'|^2-3t^2)}{(e^{2\varphi}+|\varphi'|^2+t^2)^3}
    ,\frac{-4te^{\varphi}(3e^{2\varphi}-(|\varphi'|^2+t^2))}{(e^{2\varphi}+|\varphi'|^2+t^2)^3}\bigg)
\end{equation}

\begin{equation}\label{eq:partialSN}
\begin{split}
    \partial_sN=&\bigg(\gamma'\bigg(\frac{-4e^{2\varphi}[-\varphi''(e^{2\varphi}+t^2) -2|\varphi'|^4 +|\varphi'|^2(3\varphi'' +2e^{2\varphi} -2t^2) ]}{(e^{2\varphi}+|\varphi'|^2+t^2)^3} - \frac{4t\kappa e^{2\varphi}}{(e^{2\varphi}+|\varphi'|^2+t^2)^2}\bigg)
    \\&+i\gamma'\bigg(\frac{-8t\varphi'e^{2\varphi}(2 \varphi'' - |\varphi'|^2 + e^{2 \varphi} - t^2)}{(e^{2\varphi}+|\varphi'|^2+t^2)^3}  + \frac{4\varphi'\kappa e^{2\varphi}}{(e^{2\varphi}+|\varphi'|^2+t^2)^2}\bigg),
    \\&\frac{-2 e^{\varphi} \varphi' ((6 e^{2\varphi} - 2 t^2) \varphi'' + |\varphi'|^4 - 2 |\varphi'|^2 (\varphi'' + 3 e^{2\varphi} - t^2) - 6 t^2 e^{2\varphi} + e^{4\varphi} + t^4)}{(e^{2\varphi}+|\varphi'|^2+t^2)^3}\bigg)
\end{split}
\end{equation}

Recalling that the only non-vanishing Christoffel symbols are given by $\Gamma^i_{3i} = \Gamma^i_{i3}=\frac{-1}{z}, \Gamma^3_{ii}=\frac{1}{z}, \Gamma^3_{33}=-\frac1z$, we use (\ref{eq:partialTN}), (\ref{eq:partialSN}) to calculate the following covariant derivatives

\begin{equation}\label{eq:TCovariantN}
\begin{split}
    \nabla_T N & = \partial_t N + (\Gamma^k_{ij}N^iT^j)e_k
    \\&=\bigg(\gamma'\bigg( \frac{-16t\varphi'e^{2\varphi}}{(e^{2\varphi}+|\varphi'|^2+t^2)^3} + \frac{8t\varphi'e^{2\varphi}}{(e^{2\varphi}+|\varphi'|^2+t^2)^3} + \frac{4t\varphi'(e^{2\varphi}-(|\varphi'|^2+t^2))}{(e^{2\varphi}+|\varphi'|^2+t^2)^3} \bigg) 
    \\&i\gamma'\bigg(\frac{4e^{2\varphi}(e^{2\varphi}+|\varphi'|^2-3t^2)}{(e^{2\varphi}+|\varphi'|^2+t^2)^3} + \frac{8t^2e^{2\varphi}}{(e^{2\varphi}+|\varphi'|^2+t^2)^3} - \frac{2(e^{2\varphi} + |\varphi'|^2-t^2)(e^{2\varphi}-(|\varphi'|^2+t^2))}{(e^{2\varphi}+|\varphi'|^2+t^2)^3} \bigg),
    \\&\frac{-4te^{\varphi}(3e^{2\varphi}-(|\varphi'|^2+t^2))}{(e^{2\varphi}+|\varphi'|^2+t^2)^3} + \frac{8te^{2\varphi}(e^{2\varphi}-(|\varphi'|^2+t^2))}{(e^{2\varphi}+|\varphi'|^2+t^2)^3} \bigg)
    \\&=\bigg(\gamma'\frac{-4t\varphi'}{(e^{2\varphi}+|\varphi'|^2+t^2)^2} + i\gamma' \frac{2(e^{2\varphi} + |\varphi'|^2-t^2)}{(e^{2\varphi}+|\varphi'|^2+t^2)^2}, \frac{-4te^{\varphi}}{(e^{2\varphi}+|\varphi'|^2+t^2)^2}\bigg)
    \\&=T
\end{split}
\end{equation}

\begin{align*}
    \nabla_S N & = \partial_s N + (\Gamma^k_{ij}N^iS^j)e_k
    \\&= \bigg(\gamma'\bigg(\frac{-4e^{2\varphi}[-\varphi''(e^{2\varphi}+t^2) -2|\varphi'|^4 +|\varphi'|^2(3\varphi'' +2e^{2\varphi} -2t^2) ]}{(e^{2\varphi}+|\varphi'|^2+t^2)^3} - \frac{4t\kappa e^{2\varphi}}{(e^{2\varphi}+|\varphi'|^2+t^2)^2}
    \\&+\frac{4|\varphi'|^2e^{2\varphi}(2\varphi''-|\varphi'|^2+e^{2\varphi}-t^2)}{(e^{2\varphi}+|\varphi'|^2+t^2)^3}
    \\&+\bigg(\frac{2 (e^{2\varphi} + t^2) \varphi'' - 2 |\varphi'|^2 (\varphi'' + 2 e^{2\varphi})}{(e^{2\varphi}+|\varphi'|^2+t^2)^2} - \frac{2t\kappa}{e^{2\varphi} + |\varphi'|^2+t^2} + 1\bigg)\frac{-(e^{2\varphi}-(|\varphi'|^2+t^2))}{e^{2\varphi}+|\varphi'|^2+t^2}\bigg)
    \\&+i\gamma'\bigg(\frac{-8t\varphi'e^{2\varphi}(2 \varphi'' - |\varphi'|^2 + e^{2 \varphi} - t^2)}{(e^{2\varphi}+|\varphi'|^2+t^2)^3}  + \frac{4\varphi'\kappa e^{2\varphi}}{(e^{2\varphi}+|\varphi'|^2+t^2)^2}
    \\&\frac{4t\varphi'e^{2\varphi}(2\varphi''-|\varphi'|^2+e^{2\varphi}-t^2)}{(e^{2\varphi}+|\varphi'|^2+t^2)^3}
    \\&\bigg(\frac{-4t\varphi'(\varphi''+e^{2\varphi})}{(e^{2\varphi}+|\varphi'|^2+t^2)^2}  + \frac{2\varphi'\kappa}{e^{2\varphi} + |\varphi'|^2+t^2} \bigg)\frac{-(e^{2\varphi}-(|\varphi'|^2+t^2))}{e^{2\varphi}+|\varphi'|^2+t^2}\bigg),
    \\&\frac{-2 e^{\varphi} \varphi' ((6 e^{2\varphi} - 2 t^2) \varphi'' + |\varphi'|^4 - 2 |\varphi'|^2 (\varphi'' + 3 e^{2\varphi} - t^2) - 6 t^2 e^{2\varphi} + e^{4\varphi} + t^4)}{(e^{2\varphi}+|\varphi'|^2+t^2)^3}
    \\&+ \frac{4\varphi'e^{\varphi}(e^{2\varphi}-(|\varphi'|^2+t^2))(2\varphi''-|\varphi'|^2+e^{2\varphi}-t^2)}{((e^{2\varphi}+|\varphi'|^2+t^2)^3)} \bigg)
\end{align*}
which reduces to

\begin{equation}\label{eq:SCovariantN}
\begin{split}
    \nabla_S N&=\bigg(\gamma'\bigg(\frac{2\varphi''(e^{2\varphi}-|\varphi'|^2+t^2)}{(e^{2\varphi}+|\varphi'|^2+t^2)^2} - \frac{2t\kappa}{(e^{2\varphi}+|\varphi'|^2+t^2)} - \frac{(e^{2\varphi}-(|\varphi'|^2+t^2))}{(e^{2\varphi}+|\varphi'|^2+t^2)}  \bigg)
    \\& i\gamma'\bigg( -\frac{4t\varphi'\varphi''}{(e^{2\varphi}+|\varphi'|^2+t^2)^2} + \frac{2\varphi'\kappa}{(e^{2\varphi}+|\varphi'|^2+t^2)} \bigg),
    \\&-\frac{4\varphi'\varphi''e^{\varphi}}{(e^{2\varphi}+|\varphi'|^2+t^2)^2} +\frac{2\varphi'e^\varphi}{(e^{2\varphi}+|\varphi'|^2+t^2)} \bigg)
\end{split}
\end{equation}

We proceed to calculate the first fundamental form of the caterpillar region in $(s,t)$ coordinates, by taking inner products of (\ref{eq:Tvector}) and (\ref{eq:Svector})

\begin{equation}\label{eq:1stfundform}
\begin{split}
    I(T,T) = \langle T,T\rangle &=e^{-2\varphi}\\
    I(T,S) = \langle T,S\rangle &=(k-t)\varphi'e^{-2\varphi}\\
    I(S,S) = \langle S,S\rangle &=\frac{e^{-2\varphi}}{4}\bigg(4|\varphi'|^2(k-t)^2 + (e^{2\varphi} + t(-2k+t) - |\varphi'|^2 + 2\varphi'')^2 \bigg)
\end{split}
\end{equation}

For the second fundamental form $\II$ we take inner products of (\ref{eq:Tvector}), (\ref{eq:Svector}) with (\ref{eq:TCovariantN}), (\ref{eq:SCovariantN}) to obtain

\begin{equation}\label{eq:2ndfunform}
\begin{split}
    \II(T,T) = -\langle T,\nabla_T N\rangle &=-e^{-2\varphi}\\
    \II(T,S) = -\langle T,\nabla_S N\rangle &=e^{-2\varphi}(t-k)\varphi'\\
    \II(S,S) = -\langle S,\nabla_S N\rangle &= \frac{e^{-2 \varphi}}{4} \bigg(e^{4\varphi}-4 k^2 t^2+4 k t^3-t^4-|\varphi'|^4-\\&2 |\varphi'|^2 (2 k^2-2 k t+t^2-2\varphi'')+4 (2 k-t) t \varphi''-4 |\varphi''|^2\bigg)
\end{split}
\end{equation}

The orientation of the base $T,S,N$ is given by the sign of $\langle T\times S,N \rangle$, which reduces to

\[\rm{sign}\bigg(-\frac{4 e^\varphi (e^{2 \varphi}-2 k t+t^2-|\varphi'|^2+2\varphi'')}{(e^{2\varphi}+t^2+|\varphi'|^2)^3}\bigg) = \rm{sign} (-e^{2\varphi}+2 k t-t^2+|\varphi'|^2- 2\varphi'')
\]

Then the (signed) area form can be calculated by taking the (signed) square root of the determinant of (\ref{eq:1stfundform})

\[\sqrt{\frac14 e^{-4\varphi} (e^{2\varphi}-2 k t+t^2-|\varphi'|^2+ 2\varphi'')^2}dtds = \frac12 e^{-2\varphi} (-e^{2\varphi}+2 k t-t^2+|\varphi'|^2- 2\varphi'')dtds
\]

The mean curvature $H=\frac12tr(\I^{-1}\II)$ 

\begin{equation}\label{eq:MeanCurvature}
    H = \frac{-(t(-2k+t) - |\varphi'|^2 + 2\varphi'')}{e^{2\varphi} + t(-2k+t) - |\varphi'|^2 + 2\varphi''}
\end{equation}

Principal curvatures $k_1, k_2$ given by 

\[k_1=-1,\quad k_2 = \frac{e^{2\varphi} - t(-2k+t) + |\varphi'|^2 - 2\varphi''}{e^{2\varphi} + t(-2k+t) - |\varphi'|^2 + 2\varphi''}
\]

\subsection{W-volume for domains with circular boundary}\label{subsec:W-vol}
The explicit way to show the affine relation (\ref{eq:affinerelation}) between $W$-volume and $\log\det\Delta$ is to show that their derivatives are equal up to a multiplicative factor of $3\pi$. Hence it is useful to have a formula for the derivative of $vol-\frac12\int Hda$. This is given to us by applying the (piecewise smooth) Schl\"afli formula (see \cite[Theorem 1]{RivinSchlenker}, \cite[Theorem 4]{Souam}) as done for \cite[Equation (41)]{KrasnovSchlenker08}. Namely, if we have a domain $int(\Sigma)$ in $\mathbb{H}^3$ bounded by a piecewise smooth surface $\Sigma$ that we perturb smoothly along a vector field $V$, then the derivative of volume can be expressed as

\begin{equation}\label{eq:WSchlafli}
    \left(vol(int(\Sigma))-\frac12 \int_\Sigma Hda\right)' = \int_\Sigma \frac12\delta H + \frac14\langle \delta I, \II_0 \rangle da + \frac12\sum_i \int_{G_i} \delta\theta_i(x)dx
\end{equation}
where $\delta$ denotes the (covariant) derivative with respect to $V$, $\II_0$ is the traceless part of the second fundamental form $\II$, $G_i$ are the codimension 1 faces of $\Sigma$ in the piecewise structure, and $\theta_i$ is the exterior dihedral angle function in $G_i$ define by the angle between the two faces of $\Sigma$ meeting at $G_i$. While $\theta_i$ is only defined up to an integer multiple of $2\pi$, its derivative $\delta\theta_i$ is real-valued. This is a generalization of the classical Schl\"afli formula for volume of polyhedra in space forms.

Now as we take a one parameter family of functions $\varphi(s,u)$, we are interested in its contribution towards the first integrand of the Schlafli formula, $2\delta H + \langle \delta I, \II_0 \rangle$. Hence we can derivate (\ref{eq:1stfundform}) and (\ref{eq:MeanCurvature}), take the inner product of the second term with $\II_0$ to obtain

\[2\delta H + \langle \delta I, \II_0 \rangle = \frac{4 e^{2\varphi} \varphi_u}{e^{2\varphi}-2 k t+t^2-|\varphi'|^2+2 \varphi''} 
\]

Hence the integrand $\int \frac12\delta H + \frac14\langle \delta I, \II_0 \rangle da$ can be expressed as

\[ \int \frac12\delta H + \frac14\langle \delta I, \II_0 \rangle da = -\frac12 \int  \phi_u dtds
\]

Restricting then to a domain $U$ whose boundaries is the union of finitely many disjoint round disk, then boundary component contributes

\begin{equation}\label{eq:WSchlafliCaterpillar}
\int \frac12\delta H + \frac14\langle \delta I, \II_0 \rangle da = -\frac12 \int_\gamma\int_0^{k(s)}  \phi_u dtds = -\frac12 \int_\gamma k(s)\phi_u ds
\end{equation}

\begin{defi}\label{defi:piecewisesphere}
Let $U$ be a domain in $\mathbb{C}$ so that $\partial U=\cup_{1\leq i\leq k}C_i$ is the union of finitely many disjoint circles $C_i$. Let $e^{2\varphi}|dz|^2$ be a $C^{2,\alpha}$ conformal metric in $\overline{U}$, and let $\lbrace P_i\rbrace$ be the collection of disjoint geodesic planes in $\mathbb{H}^3$ so that $\partial P_i=C_i$. We define
\begin{enumerate}
    \item $E_\varphi$ as the Epstein map $\Eps_\varphi : \overline{U}\rightarrow\mathbb{H}^3$
    \item $C_\varphi$ as a (leaf-wise) reparametrized restriction of the caterpillar region $\Eps_{\varphi|_{\partial U},\partial U}$ so that $C_\varphi:\partial U \times [0,1]\rightarrow \mathbb{H}^3$ satisfies that $C_\varphi(\partial U\times\lbrace1\rbrace)$ agrees with $E_\varphi|_{\partial U}$ (inducing the opposite orientation in the shared boundary) and $C_\varphi(\partial U\times\lbrace0\rbrace)$ belongs to the union of geodesic planes $\cup_{1\leq i\leq k} P_i$.
    \item $T_\varphi$ as the map $T_\varphi:\cup_{1\leq i\leq k}\mathbb{D}\rightarrow \cup_{1\leq i\leq k}P_i$ so that $T_\varphi|_{D_i}$ parametrizes the region interior to the closed curve $C_\varphi(C_i \times\lbrace0\rbrace)$, inducing the opposite orientation in the shared boundary.
\end{enumerate}
Finally, define the piecewise smooth map $S_\varphi:\mathbb{S}^2\rightarrow\mathbb{H}^3$ by gluing the maps $E_\varphi,C_\varphi,T_\varphi$. 
\end{defi}

Each caterpillar region in $C_\varphi$ meets its adjacent totally geodesic face of $T_\varphi$ orthogonally, while all the caterpillar regions meet the Epstein surface $E_\varphi$ tangentially.

Since $\mathbb{H}^3$ is contractible, the map $S_\varphi:\mathbb{S}^2\rightarrow\mathbb{H}^3$ admits a piecewise extension $f:\overline{B^3}\rightarrow\mathbb{H}^3$, and any two such extensions $f_0, f_1$ are homotopic through piecewise smooth extensions $\lbrace f_t\rbrace_{0\leq t\leq1}$. By Stokes theorem, the volume interior to $S_\varphi$
\begin{equation}
    vol(int(S_\varphi)) := \int_{\overline{B^3}} f^*(dvol_{\mathbb{H}^3})
\end{equation}
is well-defined. Similarly, we can use the differential of $E_\varphi, C_\varphi, T_\varphi$ to define area forms $da_{E_\varphi}, da_{C_\varphi}, da_{T_\varphi}$ in the respective domains. We denote then as $da_{S_\varphi}$ as the piecewise defined form given by these three `area' forms. As computed for instance in \cite{KrasnovSchlenker08}, if $H$ denotes the mean curvature at immersed point of $S_\varphi$, the form $Hda_{S\varphi}$ extends piecewise to the critical points of $S_\varphi$.

Now we proceed to define $W$-volume in domains with round boundary.

\begin{defi} Let $U$ be a domain in $\mathbb{C}$ so that $\partial U$ is a disjoint union of round circles. Let $e^{2\varphi}|dz|^2$ be a smooth conformal metric in $\overline{U}$. Then considering the notation introduced by Definition \ref{defi:piecewisesphere} and the following paragraph, we define $W(e^{2\varphi}|dz|^2)$ as
\begin{equation}\label{eq:Wvol}
    W(e^{2\varphi}|dz|^2) := vol(int(S_\varphi)) - \frac12 \int_{\mathbb{S}^2} Hda_{S\varphi} - \frac34 \int_{\partial U\times[0,1]} da_{C\varphi},
\end{equation}
\end{defi}
By abuse of notation, we will refer to $\int_{\mathbb{S}^2} Hda_{S\varphi}$ and $\int_{\partial U\times[0,1]} da_{C\varphi}$ as $\int_{S_\varphi}Hda$ and $\int_{C_\varphi}da$, respectively.

Observe that if $U$ is a fundamental domain of a Kleinian group $\Gamma$ and $e^{2\varphi}|dz|^2$ is a $\Gamma$ invariant metric, the $W$-volume $W(e^{2\varphi}|dz|^2)$ agrees with the traditional $W$-volume of the quotient manifold. This is an easy statement to verify, since the extra terms coming from $\partial U$ cancel between the components identified by some element of $\Gamma$, so the terms in (\ref{eq:Wvol}) reduce to the terms appearing in (\ref{eq:OGWvol}) for $r=0$.

Now we are ready to prove that our definition of $W$-volume satisfies the affine relation (\ref{eq:affinerelation}) with $\log\det\Delta$

\begin{theorem}[Theorem \ref{thm:main}]
Let $U$ be a domain in $\mathbb{C}$ so that $\partial U$ is a disjoint union of round circles. Let $g$ be a smooth conformal metric in $\overline{U}$ and $\varphi:\overline{U}\rightarrow\mathbb{R}$ be a smooth function. The the $W$-volume satisfies\footnote{For the first equality it is enough to ask for $\varphi$ to belong to $C^{2,\alpha}(\overline{U})$, but the second equality requires $C^\infty$ smoothness}
\begin{equation}\label{eq:Polyakov}
\begin{split}
    W(e^{2\varphi}g) - W(g) &= -\frac14\int_U |\nabla_g\varphi|^2 + Scal(g)\varphi da(g) -\frac12\int_{\partial U} k(g)\varphi ds(g) -\frac34 \int_{\partial U} \partial_n\varphi ds(g)
    \\&= 3\pi (\log\det\Delta(e^{2\varphi}g) - \log\det\Delta(g)),
\end{split}
\end{equation}
where $\nabla$ is the gradient, $Scal(.)$ is the scalar curvature, $k(.)$ is the geodesic curvature, $\partial_n$ is the derivative with respect to the inward pointing normal $n$, $da, ds$ are the area and length forms, respectively. The second equality follows from \cite[Equation (1.17)]{OsgoodPhillipsSarnak}
\end{theorem}

\begin{proof}
Equivalently, we will show that

\begin{equation}\label{eq:Wvolderivative}
    \frac{d}{dt} W(g_t:=e^{2t\varphi}g) = -\frac14\int_U  Scal(g_t)\varphi da(g_t) -\frac12\int_{\partial U} k(g_t)\varphi ds(g_t) -\frac34 \int_{\partial U} \partial_{n_t}\varphi ds(g_t).
\end{equation}
Then the identity \ref{eq:Polyakov} (also known as the Polyakov-Alvarez formula) follows by integrating (\ref{eq:Wvolderivative}) from $t=0$ to $t=1$ and using the identities $Scal(g_t) = e^{-2t\varphi}(Scal(g) - 2t\Delta\varphi)$, $da(g_t) = e^{2t}da(g)$, $k(g_t)=e^{-t\varphi}(k(g)+t\partial_n\varphi)$, $\partial_{n_t}\varphi = e^{-t\varphi}\partial_{n}\varphi$, $ds(g_t)=e^{t\varphi}ds(g)$. Given the generallity of (\ref{eq:Wvolderivative}), it is enough to prove it for $t=0$.

Let us then break down this derivative as

\begin{equation}\label{eq:derivativebreakdown}
\frac{d}{dt}\bigg\vert_{t=0} W(g_t) = \left(vol(int(S(e^{2\varphi}|dz|^2))) -\frac12 \int_{S_\varphi} Hda\right)' - \left(\frac34 \int_{C_\varphi} da\right)'
\end{equation}

Using (\ref{eq:WSchlafli}), the first term of (\ref{eq:derivativebreakdown}) can be expressed as

\begin{equation}
    = \left(\int_{E(g)} + \int_{T(g)} + \int_{C(g)}\right) \frac12\delta H + \frac14\langle \delta I, \II_0 \rangle da + \frac12\left(\int_{C(g)\cap E(g)} + \int_{C(g)\cap T(g)}\right) \delta\theta(x)dx.
\end{equation}
Since the angles between the faces of $S(g)$ are constant equal to either $0,\pi/2$, all terms with $\delta\theta$ vanish. Because each component of $T(g)$ is totally geodesic, the integral  $\int_{T(g)}\frac12\delta H + \frac14\langle \delta I, \II_0 \rangle da$ vanishes as well. By \cite[Corollary 6.2]{KrasnovSchlenker08} we see that $\int_{E(g)}\frac12\delta H + \frac14\langle \delta I, \II_0 \rangle da$ is equal to $-\frac14\int_U  Scal(g_t)\varphi da(g_t)$, while by (\ref{eq:WSchlafliCaterpillar}) we have that $\int_{C(g)}\frac12\delta H + \frac14\langle \delta I, \II_0 \rangle da$ is equal to $-\frac12 \int_\gamma k(s)\phi_u ds$. Hence the first term of (\ref{eq:derivativebreakdown}) reduces to

\begin{equation}\label{eq:breakdownpart1}
= -\frac14\int_U  Scal(g_t)\varphi da(g_t) -\frac12\int_{\partial U} k(g_t)\varphi ds(g_t)
\end{equation}

For the second term of (\ref{eq:derivativebreakdown}) it is rather easier to directly compare $-\frac34 Area(e^{2\varphi}g)$ and $-\frac34 Area(g)$ as

\begin{align*}
    Area(g) &= \int_{\partial U}\int_0^{k(g)(s)}dtds = \int k(g)ds(g)\\
    Area(e^{2\varphi}g) &= \int_{\partial U}\int_0^{k(e^{2\varphi}g)(s)}dtds = \int k(e^{2\varphi}g)ds(e^{2\varphi}g),
\end{align*}
so then by using that $k(e^{2\varphi}g)ds(e^{2\varphi}g) = (k(g) + \partial_n\varphi)ds(g)$ we have

\[
Area(e^{2\varphi}g) - Area(g) =  \int_{\partial U} \partial_n\varphi ds(g)
\]
from where it follows that the second term of (\ref{eq:derivativebreakdown}) equals to

\begin{equation}\label{eq:breakdownpart2}
-\frac34\int_{\partial U} \partial_n\varphi ds(g)
\end{equation}

Hence by combining (\ref{eq:breakdownpart1}) and (\ref{eq:breakdownpart2}) we have that (\ref{eq:derivativebreakdown}) follows.

\end{proof}

\section{Applications}\label{sec:applications}

\subsection{Renormalized volume of classical Schottky groups}\label{subsec:negVR}

By \cite[Theorem 1]{OsgoodPhillipsSarnak} we know that under the condition $\int_{\partial U} kds\geq 0$, the conformal metric $g_0$ of maximum determinant is given by the hyperbolic metric in $U$ with totally geodesic boundary. Observe also that any $\Gamma$-invariant metric $e^{2\varphi}|dz|^2$ (for some Kleinian group $\Gamma$) satisfies $\int_{\partial U} kds = 0$, and the conformal metric $g_0$ is the hyperbolic metric invariant by the group $\Gamma_0$ generated by the reflections along the geodesic planes at each component of $\partial U$, so identified summands in $\int_{\partial U} kds$ cancel.

Our first application uses this maximality, that in particular concludes that every Riemann surface has a Schottky uniformization with negative $\VR$ if a classical Schottky uniformization has sufficiently thick fundamental domain.

\begin{theorem}[Theorem \ref{thm:main1}]
Let $M$ be the hyperbolic handlebody obtained by the classical Schottky uniformization of a surface $\Sigma$ of genus $g$. Then $\VR(M) < (6g-8)\pi$. Moreover, if the distance between geodesic planes on the boundary of a fundamental domain for $M$ is at least $4$, then $\VR(M)<-2\pi$.
\end{theorem}
\begin{proof}
Let us first aim to show the second inequality.

Take $U\subset \mathbb{C}$ a fundamental domain for $M= \mathbb{H}^3/\Gamma$ for that $\partial U$ is the union of finitely many disjoint round circles. Let then $e^{2\varphi}|dz|^2$ be the hyperbolic $\Gamma$-invariant metric in $U$, and as we discussed at the start of this section, let $g_0$ be the conformal $\Gamma_0$-invariant metric in $U$ of maximal determinant, where $\Gamma_0$ is the group generated by the reflections along the geodesic planes at each component of $\partial U$. Then by \cite[Theorem 1]{OsgoodPhillipsSarnak} we have that
\begin{equation}
    \VR(M)=W(e^{2\varphi}|dz|^2) \leq W(g_0),
\end{equation}
so it is sufficient to prove negativity of $W(g_0)$. For this, we will compare $W(g_0)$ against the $W$-volume of the convex core associated to $g_0$.

One can describe the convex core of $g_0$ as follows. We will denote each component of $\partial U$ by a capital letter $C$, while we will denote by $P$ the geodesic plane with boundary at infinity equal to $C$. For each pair of components $C_{1,2}$ of $\partial U$ we can take the geodesic segment $\alpha_{1,2}$ that realizes the distance between the geodesic planes $P_{1,2}$ with boundary at infinity equal to $C_{1,2}$. For each triple of components $C_{\lbrace 1,2,3\rbrace}\subset \partial U$ we have a right-angled hexagon $H_{\lbrace 1,2,3\rbrace}$ made out of the segments $\alpha_{12},\alpha_{13},\alpha_{23}$ and the segments $\beta_{1,2,3}$ in $P_{1,2,3}$ joining the endpoints of the adjacent $\alpha$'s. This 6 segments belong to the unique geodesic plane orthogonal to $P_1, P_2, P_3$.

At each component $C\subset \partial U$ we consider all vertices coming from the $\alpha$ segments and edges given by the $\beta$ segments. Hence in $P$ we have a maximal convex polygon $P_C$, which can be joined to other convex polygons $P_{C'}$ by its adjacent right-angled hexagons $H$'s. Then the convex polygons $P_C$ and their adjacent right-angled hexagons $H$'s assemble into the boundary of a convex polyhedra, which is the convex core $\mathcal{C}$ of $\Gamma_0$ when we add the faces $P_C$ with multiplicity 2.

In order to define $W$-volume of $\mathcal{C}$ we have two equivalent options. We either consider directly the formula (\ref{eq:Wvol}) using that the boundary term vanish and $\int 2Hda$ is given by the sum 
\[\sum_e \theta(e)\ell(e)
\]
where $e$ runs along the edges of $\mathcal{C}$ adjacent to two right-angled hexagons, $\theta(e),\ell(e)$ are the exterior dihedral angle and length of $e$, respectively.

Equivalently, we can take a $C^{1,1}$ metric in $U$ define by the first contact horosphere (so that the previous construction recovers $\mathcal{C}$) and give the associated $W$-volume. It could also be helpful to visualize instead the torsion free degree 2 cover of $\Gamma_0$.

Regardless, by monotonicity of $W$-volume for non-positive metrics (using either \cite[Proposition 3.11]{Schlenker} for $W$-volume or \cite[Equation 1.17]{OsgoodPhillipsSarnak} for $\log\det\Delta$) we have that
\[W(g_0) < W(\mathcal{C}),
\]
so if we want to show that $\VR(M)<0$ it is sufficient to show $W(\mathcal{C})$ is negative, or equivalently, that $vol(\mathcal{C})\leq \frac14\sum_e \theta(e)\ell(e)$. Taking $\ell$ to be the minimum distance between the geodesic planes resting on top of $\partial U$, we can even further reduce to prove $vol(\mathcal{C})\leq \frac\ell4\sum_e \theta(e)$. Recall that the domain $U$, $\partial U$ has $2g$ components ($g>1$).

For each convex polygon $P_C$ in $C\subseteq \partial U$ we have by the orthogonality at $P_C$ and Gauss-Bonnet
\begin{equation}\label{eq:anglerelabel}
    \sum_{e \text{ adjacent to } P_C} \theta(e) = Area(P_C)+2\pi
\end{equation}

Since the right-angled hexagon faces of $\mathcal{C}$ assemble into a $2g$ holed sphere, we have that

\begin{equation}
    Area(\partial C) = 2\pi(2g-2) + \sum_{C\subset\partial U} Area(P_C)
\end{equation}

Combining this with the isoperimetric inequality, we have

\begin{equation}
    vol(\mathcal{C}) < \frac12Area(\partial C) = \pi(2g-2) + \sum_{C\subset\partial U} \frac12 Area(P_C)
\end{equation}

Replacing (\ref{eq:anglerelabel}) for each of the $n$ components of $\partial U$ and using that each edge $e$ is adjacent to exactly 2 planes, we get

\begin{equation}
    vol(\mathcal{C}) < \pi(2g-2) + \left(\sum_e \theta(e)\right) - 2g\pi  = \left(\sum_e \theta(e)\right) -2\pi
\end{equation}

Hence the desired inequality $vol(\mathcal{C})\leq \frac\ell4\sum_e \theta(e)$ would follow if $\ell\geq 4$. In such case, we will have then that $\VR(M)<-2\pi$.

For the first inequality, we still have $\VR(M)< vol(\mathcal{C}) < \left(\sum_e \theta(e)\right) -2\pi$. By Euler characteristic we have that there are $6g-6$ edges, and each exterior dihedral angle is bounded by $\pi$. Hence it follows that $\VR(M)<(6g-6)\pi-2\pi=(6g-8)\pi$.
\end{proof}

\begin{remark}
The attentive reader will observe that we could have ran the argument using $W$-volume along metrics invariant by some boundary identifications, where the identifications are allowed to change. This would have not required the use of boundary terms since they will always cancel throughout these path of metrics. Regardless, it is of independent interest to realize geometrically in $\mathbb{H}^3$ the determinant of the Laplacian. Addressing a well-known maxima, our construction `draws' the `full glissando'\footnote{\emph{Glissando} is the continuous slide between two musical notes. Here we are loosely using the term \emph{full glissando} to symbolize that $\log\det\Delta$ is `adding' the logarithm of all possible frequencies.} of a drum with round edges.
\end{remark}

\subsection{Loewner energy of $C^{2,\alpha}$ Jordan curves}\label{subsec:Loewner}

As it was observed to us by Yilin Wang, we can use our description to describe the Loewner energy of a Weil--Petersson Jordan curve. 

Let $\gamma$ be a $C^{2,\alpha}$ Jordan curve. Take $f_1:\mathbb{D}\rightarrow\overline{\mathbb{C}}, f_2:\mathbb{D}^*\rightarrow\overline{\mathbb{C}}$ uniformization maps of the components of $\overline{\mathbb{C}}\setminus\gamma$, that extend up to second order to $\gamma$. Then we can define conformal metrics $g_{1,2} = \frac{4e^{2\varphi}}{(1+|z|^2)^2}|dz|^2$ on $\mathbb{D}, \mathbb{D}^*$ by taking $f_{1,2}^*(g_{\mathbb{S}^2})$, where $g_{\mathbb{S}^2}$ is the round metric in $\mathbb{S}^2$. Hence the Loewner energy of $\gamma$ can be calculated as (see proof of \cite[Theorem 7.3]{Wang})

\begin{equation}
    I^L(\gamma) = 12(\log\det(g_{\mathbb{S}^2}|_{\mathbb{D}}) + \log\det(g_{\mathbb{S}^2}|_{\mathbb{D}^*})- \log\det(g_1)- \log\det(g_2)) 
\end{equation}
It is a simple exercise to verify that the Epstein maps for $g_{1,2}$ have constant image equal to $(0,0,1)$, which belongs to the totally geodesic plane with boundary $\mathbb{S}^1=\partial\mathbb{D}=\partial\mathbb{D^*}$. Hence both piecewise-smooth maps associated to $g_{1,2}$ for the definition of $W$-volume are constant, from where it follows that $W(g_{\mathbb{S}^2}|_{\mathbb{D}}) = W(g_{\mathbb{S}^2}|_{\mathbb{D}^*})=0$. Hence by applying the affine relation (\ref{eq:affinerelation}) of Theorem \ref{thm:main} we obtain that
\[ I^L(\gamma) = \frac{4}{\pi}(W(g_{\mathbb{S}^2}|_{\mathbb{D}}) - W(g_1) + W(g_{\mathbb{S}^2}|_{\mathbb{D}^*}) - W(g_2)) = -\frac{4}{\pi}(W(g_1)+W(g_2)),
\]
which is the conclusion of Theorem \ref{thm:main2}.

\bibliographystyle{amsalpha}
\bibliography{mybib}

\end{document}